\documentclass{article}

\usepackage[english]{babel}
\usepackage{amsmath,amsthm}
\usepackage{amsfonts}
\usepackage{graphicx}
\usepackage{tikz}
\usepackage{hyperref}
\usepackage{titlesec}
\usepackage{thmtools}

\newtheorem{theorem}{Theorem}

\numberwithin{equation}{section}

\title{Pyramidal Polytopes in the Stability Region}

\author{Vakif Dzhafarov (Cafer) \\ Department of Mathematics, Faculty of Science,\\
Eskisehir Technical University, Eskisehir 26470, Turkey \\ 
\"{O}zlem Esen \\ School for the Handicapped,\\ Anadolu University, Eskisehir 26470, 
Turkey \\ Taner B\"uy\"ukk\"oro\u{g}lu\\ Department of Mathematics, Faculty of Science,\\
Eskisehir Technical University, Eskisehir 26470, Turkey} 
\begin{document}
\maketitle
\begin{abstract}
Every $n$th order monic polynomial corresponds $n$-dimensional vector. If the given polynomial is stable that is all its roots lie in the open left half plane it is said to be Hurwitz polynomial and the corresponding vector is called stable vector. The set of stable vectors is non-convex. In this paper, we define special $(n+1)$ stable vectors such that their convex hull is stable. 
\end{abstract}



\section{Introduction}
Consider the following polynomial
\begin{equation}
\label{polinom}
a(s)=a_1+a_2s+...+a_ns^{n-1}+s^n
\end{equation}
The  polynomial $a(s)$ corresponds $n$-dimensional vector $a=(a_1,a_2,...,a_n)^T \in \mathbb{R}^{n}$, where the symbol ``T'' stands for the transpose, and the symbol $R^n$ stands for the $n$-dimensional real Euclidean space. 
\\ The polynomial (\ref{polinom}) has $n$ roots and if all roots $s_1,s_2,...,s_n$ satisfy the condition $Re(s_i)<0$ $(i=1, 2, ..., n)$ then it is called (Hurwitz) stable \cite{Barmish1994,Bhattacharyya1995}.
\\ Define 
\[
  \mathcal{H}^{n}=\{a=(a_1,a_2,\dots,a_n)^T \in \mathbb{R}^n: \ a(s) \mbox{ is Hurwitz stable}\}.
\]

It is well known that $\mathcal{H}^{n}\subset\mathbb{R}^{n}_+$, where

\[
\mathbb{R}^{n}_+=\{x=(x_1, x_2,..., x_n)^T\in\mathbb{R}^{n}:x_{i}>0, (i=1, 2, ..., n)\}.
\]

This necessary condition is not sufficient. Indeed the polynomial \\ 
\[
a(s)=s^3+\frac{2}{3} s^2+\frac{1}{6}s+\frac{1}{2}\]
\\
is not stable since has roots $s=-1$ and $s=\frac{1}{6}\pm i\frac{\sqrt{17}}{6}$.\\ 

The convex hull of a finite number of points from $\mathbb{R}^{n}$ is called a polytope in $\mathbb{R}^{n}$. 
\\

The aim of this note is to define special polytopes in  $\mathbb{H}^{n}$. These polytopes are defined as the convex hull of $(n+1)$ points therefore they has a pyramidal structure. On of these points is especially defined point from $\mathcal{H}^{n}$, the remaining $n$ points belong to the boundary of the set $\mathcal{H}^{n}$. Therefore the interiors of these polytopes belong to $\mathcal{H}^{n}$.

\section{Polytopes in $\mathcal{H}^{n}$ }
\label{sec:headings}

Second order polynomial $s^2+k_1s+k_2$ is stable if and only if $k_1>0$ and $k_2>0$.\\ Let $n$ be even, $m=\frac{n}{2}$
and $\alpha_1, \alpha_2,..., \alpha_m$ are different positive numbers satisfying conditions $\alpha_i\geq 1$, $(i=1, 2, ..., m)$.
Consider the following $nth$ order stable polynomial
 
\[
a^0(s)=(s^2+\alpha_1 s+\alpha_1)(s^2+\alpha_2 s+\alpha_2)\cdots(s^2+\alpha_m s+\alpha_m)
\]
which corresponds to the stable point $a^0\in\mathcal{H}^{n}$. Consider
\[
  \begin{array}{rcl}
  a^1(s)&=&(s^2+\alpha_1).(s^2+ \alpha_2 s+\alpha_2)\cdots(s^2+\alpha_m s+\alpha_m)\\
 a^2(s)&=&(s^2+\alpha_1 s).(s^2+ \alpha_2 s+\alpha_2)\cdots(s^2+\alpha_m s+\alpha_m)\\
  a^3(s)&=&(s^2+\alpha_1 s+\alpha_1).(s^2+ \alpha_2)\cdots(s^2+\alpha_m s+\alpha_m)\\
 a^4(s)&=&(s^2+\alpha_1 s+\alpha_1).(s^2+ \alpha_2s)\cdots(s^2+\alpha_m s+\alpha_m) \\
    &\vdots&\\
a^n(s)&=&(s^2+\alpha_1 s+\alpha_1).(s^2+ \alpha_2 s+\alpha_2)\cdots(s^2+\alpha_m s)\\
\end{array}
\]
We set
 \begin{equation}
  \label{polconv1}
  \mathcal{P}=\mathrm{conv}\{a^0, a^1, \dots, a^n\} \subset \mathbb{R}^{n}.
\end{equation}

\begin{theorem} The interior of the polytope $\mathcal{P}$ (\ref{polconv1}) belongs to $\mathcal{H}^n$ that is the interior points of the polytope  $\mathcal{P}$ are stable.
\end{theorem}

\begin{proof} We will show that:\\
If both $i$ and $j$ are even $(1\le i, j\le n, i \neq j)$ then the segment $[a^i(s),  a^i(s)]$ completely belongs to the stability boundary. If at least one of the numbers $i$ and  $j$ is odd then the open segment $(a^i(s),  a^i(s))$ is stable. On the other hand the segments $[a^0(s),  a^i(s)]$ $(1\le i\le n)$  are stable by construction. \\ 

Given two polynomials $a^i(s)$ and  $a^j(s))$ $(1\le i, j\le n, i \neq j)$, $(m-2)$ second order factors of these polynomials are the same. Therefore it is sufficient to investigate the stability property of the following fourth order polynomial polytope  conv$\{b^1, b^2, b^3, b^4\}$, where 

\[
  \begin{array}{lcl}
  b^1(s)=(s^2+\alpha)(s^2+\beta s+\beta), \\
  
  b^2(s)=(s^2+\alpha s)(s^2+\beta s+\beta),\\
    
  b^3(s)=(s^2+\alpha s+ \alpha)(s^2+\beta),\\ 
  
  b^4(s)=(s^2+\alpha s+ \alpha)(s^2+\beta s),\\     
  \end{array}
\]
and $\alpha\geq1, \beta\geq1, \alpha\neq\beta$.
\\ Stability of the segments $(b^1(s),  b^2(s))$ and $(b^3(s),  b^4(s))$ are obvious.\\ Stability of  $(b^1(s),  b^3(s))$: We have 

\[
  \begin{array}{lcl}
  b^1(s)=s^4+ \beta s^3+(\alpha+ \beta)s^2+\alpha \beta s+\alpha \beta, \\
  b^3(s)=s^4+\alpha s^3+(\alpha+ \beta)s^2+\alpha \beta s+\alpha \beta.  
  \end{array}
\]
For $\lambda \in (0,1)$ the convex combination polynomial is

\[
  \begin{array}{lcl}
  \lambda b^1(s)+(1-\lambda )b^2(s)= s^4+[\lambda\beta+(1-\lambda)\alpha ]s^3+(\alpha+ \beta)s^2+\alpha \beta s+\alpha \beta .  
    \end{array}
\]    
Recall that the fourth order monic polynomial $b_1+b_2s+b_3s^2+b_4s^3+s^4$ 
with positive coefficients is stable if and only if

\[
  \begin{array}{lcl}
  b_2\cdot b_3 \cdot b_4 - b_1 \cdot b_4^2-b_2^2 >0.
  
   \end{array}
\]
Therefore for the above convex combination we have
 \[
  \begin{array}{lcl}
  \lambda b^1(s)+(1-\lambda )b^2(s)&=&\alpha\beta(\alpha+ \beta)[\lambda( \beta-\alpha)+\alpha]-\alpha\beta[\lambda(\beta-\alpha)+\alpha]^2-(\alpha\beta)^2\\
  &=& \alpha\beta[\lambda(\beta^2-\alpha^2)+\alpha^2+\alpha\beta-\lambda^2(\beta-\alpha)^2-2\lambda\alpha(\beta-\alpha)-\alpha^2-\alpha\beta]\\
  &=&\alpha\beta\lambda(\beta-\alpha) [\alpha+\beta-\lambda(\beta-\alpha)-2\alpha]\\
  &=&\lambda(1-\lambda) \alpha\beta(\beta-\alpha)^2>0.  
    \end{array}
\]
Stability of 
$(b^1(s),  b^4(s))$:

\[
  \begin{array}{lcl}
  \lambda b^1(s)+(1-\lambda )b^4(s)=s^4+(-\lambda\alpha+\alpha+\beta)s^3+(-\lambda\alpha\beta+\lambda\beta+\alpha\beta+\alpha)s^2+\alpha\beta s+\lambda\alpha\beta
\end{array}
\] and\\

\[
  \begin{array}{lcl}
    b_2\cdot b_3 \cdot b_4 - b_1 \cdot b_4^2-b_2^2
    &=&(\alpha^2\beta)(1-\lambda)(\alpha\beta-\lambda\alpha\beta+\alpha-\lambda\alpha+\beta^2-\lambda\beta+\lambda^2\alpha>\\
   &&(\alpha^2\beta)(1-\lambda)[(1-\lambda)\alpha\beta+\alpha(1-\lambda)+\lambda^2\alpha+\beta(1-\lambda)] >0
  \end{array}
\]
We have used the inequality $\beta^2>\beta$ which is a consequence of the condition $\beta>1$.\\ Stability of $(b^2(s),  b^3(s))$ follows from the stability proof of the segment $(b^1(s),  b^4(s))$.\\ 
Now consider the segment $[b^2(s),  b^4(s)]$. We show that for any $\lambda\in[0,1]$ the polynomial $\lambda b^2(s)+(1-\lambda )b^4(s)$ has one root $s=0$, whereas the remaining three roots as stable. Indeed, 

\[
\begin{array}{lcl}
  b^2(s)=s(s+\alpha)(s^2+\beta s+\beta)=s[s^3+(\alpha+\beta)s^2+(\alpha\beta+\beta)s+\alpha\beta],\\
  b^4(s)=s(s+\beta)(s^2+\alpha s+\alpha)=s[s^3+(\alpha+\beta)s^2+(\alpha\beta+\alpha)s+\alpha\beta],\\  \end{array}
\]

\[
  \begin{array}{lcl}
   \lambda b^2(s)+ (1-\lambda ) b^4(s)=s [s^3+(\alpha+\beta)s^2+((1-\lambda)\alpha+\lambda\beta+\alpha\beta)s+\alpha\beta]   
    \end{array}
\]
Recall that the third order monic polynomial $b_1+b_2s+b_3s^2+s^3$ with positive coefficients is stable if and only if 

\[
\begin{array}{lcl}
  b_2 b_3-b_1>0. 
  \end{array}
\] 
 For the above third order polynomial in the bracket we have  
  
  \[
\begin{array}{lcl}
  b_2 b_3-b_1=[(1-\lambda)\alpha+\lambda\beta+\alpha\beta](\alpha+\beta)-\alpha\beta\\=[(1-\lambda)\alpha+\lambda\beta](\alpha+\beta)+ \alpha\beta(\alpha+\beta-1)>0 
   \end{array}
   \]  
(We have used the conditions $\alpha>1$ and $\beta >1$). Consequently all three roots of this polynomial are stable and the fourth root is $s=0$  which belongs to the stability boundary. \\ In conclusion, the polytope $\mathcal P$ (\ref{polconv1}) has the following properties:
\begin{itemize} 
\item[-] The segments $[a^0(s), a^i(s)]$, where $i= 1, 2,...,n$, are stable by the construction. 
\item[-] If at least one of numbers $i$ and  $j$ is odd then the open segment $(a^i(s), a^j(s))$ is stable.
\item[-] If both $i$ and  $j$ are even $(i<j)$ then the segment $[a^i(s), a^j(s)]$ is placed on the stability boundary.
\end{itemize}
By the Edge Theorem the inner points of the polytope $\mathcal P$ are stable \cite{Hollot1988,Nurges2005}.
\end{proof}
Now assume than $n$ is odd and $n=2m+1$. As in the case when $n$ is even, choose arbitrary different numbers $\alpha_1, \alpha_2,..., \alpha_m, \alpha_{m+1}$ with $\alpha_i\geq1 $ $(i= 1, 2,...,m+1)$. Consider the following $n th$ order stable polynomial  
 
 \[
\begin{array}{lcl}
a^0(s)=(s^2+\alpha_1 s+\alpha_1) \cdots(s^2+\alpha_m s+\alpha_m)(s+\alpha_{m+1}) \\
\end{array}
 \]
  Define the following $n$ polynomials from the stability boundary 
 
 \[
\begin{array}{rcl}
a^1(s)&=&(s^2+\alpha_1).(s^2+ \alpha_2 s+\alpha_2) \cdots(s^2+\alpha_m s+\alpha_m)(s+\alpha_{m+1}) \\
a^2(s)&=&(s^2+\alpha_1 s).(s^2+ \alpha_2 s+\alpha_2) \cdots(s^2+\alpha_m s+\alpha_m)(s+\alpha_{m+1}) \\
    &\vdots&\\
 a^{n-2}(s)&=&(s^2+\alpha_1 s+\alpha_1) \cdots(s^2+\alpha_m)(s+\alpha_{m+1}) \\   
 a^{n-1}(s)&=&(s^2+\alpha_1 s+\alpha_1) \cdots(s^2+\alpha_m s)(s+\alpha_{m+1}) \\
a^n(s)&=&(s^2+\alpha_1 s+\alpha_1) \cdots(s^2+\alpha_m s+\alpha_m)s \\ 
 \end{array}
 \]  
 and the polytope 
  \begin{equation}
  \label{polconv2}
  \mathcal{P}=\mathrm{conv}\{a^0, a^1, \dots, a^n\}.
\end{equation}
 
\begin{theorem}
The interior points of the polytope $\mathcal{P}$ (\ref{polconv2}) are stable.
\end{theorem}

\begin{proof}- The segments 
$[a^0(s), a^i(s)), (i= 1, 2,...,n)$ are stable by the construction. \\ -The segments $[a^i(s), a^j(s)]$, where $i$ and  $j$ are even and $2\leq i,j \leq n-1$, are placed on the stability boundary (see  the segment $[b^1(s),  b^4(s)]$from the proof of Theorem 1 ).\\
- If $1\leq i,j \leq n-1$ and least one of numbers $i$ and  $j$ is odd then the stability of  $(a^i(s), a^j(s))$ can be proven by the analogy with the proof of Theorem 1. \\ Consider the polytope  conv$\{c^1(s), c^2(s), c^3(s)\} $, where 
\[
\begin{array}{lcl}
  
  c^1(s)=(s^2+\alpha)(s+\beta)\\
  c^2(s)=(s^2+\alpha s)(s+\beta)\\
  c^3(s)=(s^2+\alpha s+\alpha)s
  \end{array}
\]

This polytope has three edges. The edge $(c^1(s), c^2(s))$ is stable.\\ Stability of $(c^1(s), c^3(s))$:
\[
\begin{array}{lcl}
  
  c^1(s)=s^3+\beta s^2+\alpha s+ \alpha\beta, \\
  
  c^3(s)=s^3+\alpha s^2+\alpha s,\\
  \lambda c^1(s)+(1-\lambda)c^3(s)= s^3+[\lambda\beta+(1-\lambda)\alpha]+ \alpha s+ \lambda\alpha\beta

\end{array}
\]
  and
  
  \[
\begin{array}{lcl}
   [\lambda\beta+(1-\lambda)\alpha]\alpha-\lambda\alpha\beta=(1-\lambda)\alpha^2>0,
    \end{array}
\]
and the segment $(c^1(s), c^3(s))$ is stable. \\
The segment  $(c^2(s), c^3(s))$ is placed on the stability boundary since it has the root $s=0$ and two stable roots.\\
Consequently, all edges of the polytope $\mathcal{P}$ (\ref{polconv2}) are stable or placed on the stability boundary. By the Edge Theorem in interior points of this polytope are stable. 
\end{proof}  

\end{document}